\documentclass[11pt]{amsart}
\usepackage[T1]{fontenc}
\usepackage{amssymb,amsopn,amsmath,amsthm,authblk,hyperref,cleveref,lipsum,scrextend,mathrsfs,array,graphicx,bbm,enumerate,wasysym,yfonts,ulem,dsfont,graphicx,enumitem}
\usepackage{stmaryrd}
\usepackage{subcaption}
\usepackage{verbatim}
\usepackage[all]{xy}
\usepackage{tikz}
\usepackage[displaymath, mathlines]{lineno}
\usepackage{fullpage}
\usepackage[foot]{amsaddr}
\usepackage{xifthen}
\usepackage{xstring}
\usepackage{xparse}
\usepackage{float}
\usepackage{xcolor,todonotes}
\usepackage[left=2cm,right=2cm,top=2cm,foot=2.5cm]{geometry}
\newtheorem{thm}{Theorem}[section]
\newtheorem{lemma}[thm]{Lemma}
\newtheorem{rem}[thm]{Remark}
\newtheorem{defn}[thm]{Definition}

\newtheorem{prop}[thm]{Proposition}

\newcommand{\QQ}{\mathcal{Q}}

\newcommand{\N}{\mathbb N}
\newcommand{\R}{\mathbb R}
\newcommand{\E}{\mathbb E}

\renewcommand{\P}{\mathbb P}

\renewcommand{\1}{\mathbf 1}

\newcommand{\F}{\mathcal{F}}
\newcommand{\FF}{\mathscr F}

\newcommand{\SM}{\mathsf{Q}}

\newcommand{\type}[1]{\text{Type }#1}
\newcommand{\m}{\mathfrak{m}}

\newcommand{\q}{\mathfrak q}

\newcommand{\p}{\mathfrak p}

\newcommand{\PM}{\mathcal{M}}
\newcommand{\Qm}{\mathcal Q}

\newcommand{\BM}{\mathfrak{Q}}

\newcommand{\Lalg}{\subset}

\newcommand\encircle[1]{%
  \begin{tikzpicture}
    \node[draw,circle,inner sep=0.6pt] {#1};
  \end{tikzpicture}}

\newcommand{\glue}{\, \encircle{\scalebox{.3}{g} } \,}

\def\cro#1{\llbracket#1\rrbracket}

\makeatletter
\g@addto@macro{\endabstract}{\@setabstract}
\newcommand{\authorfootnotes}{\renewcommand\thefootnote{\@fnsymbol\c@footnote}}
\makeatother

\let \epsilon \varepsilon

\numberwithin{equation}{section}

\crefformat{footnote}{#2\footnotemark[#1]#3}

\title{Characterisation of Markov properties on planar maps}

\author{Pablo Araya$^\dagger$}
\address{$^\dagger$ Universidad de Chile,  Centro de Modelamiento Matemático (AFB170001), UMI-CNRS 2807, Beauchef 851, Santiago, Chile.}

\author{Luis Fredes$^*$}
\address{$^*$Univ. Bordeaux, CNRS, Bordeaux INP, IMB, UMR 5251, F-33400 Talence, France.}
\author{Avelio Sepúlveda$^\dagger$}

\begin{document}
	\maketitle
	\begin{abstract} 
    We revisit, in a self-contained way, the Markov property on planar maps. First, we characterise the laws on planar maps that satisfy both the Markov property and invariance under rerooting along the boundary, proving that they are precisely Boltzmann laws. Second, we characterise random submaps satisfying the Markov property, demonstrating that they are not restricted to those obtained through a peeling procedure.
	\end{abstract}

 \section{Introduction} 
Random planar maps have captured the interest of the probabilistic and combinatorial communities in the 21st century. They originally appeared in physics as the natural candidates for discretisation of random  two-dimensional metric spaces. One of the key properties that enables the study of random planar maps is the Markov property, usually presented via the so-called peeling procedure. In this paper, we study the Markov property as in the folklore of Markov processes, by first showing a weak version, related to a deterministic submap and then proving it in a strong sense by defining (random) \textit{stopping maps}. This naturally leads to two questions. The first is to determine which probability laws on planar maps satisfy a Markov property. The second is, given  such a law, to characterise all the random submaps that induce a strong Markovian decomposition. In this paper we answer both of those questions.

Quadrangulations with a boundary are among the simplest and most extensively studied classes of planar maps. These maps have been extensively studied from both probabilistic and combinatorial perspectives. Explicit formulae for their enumeration are known \cite{tutte} and can be obtained through explicit bijections \cite{SCH98,BG}. Moreover their local limit \cite{krikun,CM} and their scaling limit, the Brownian map and disk, \cite{LeGall,Miermont,BetM,CLG,BGR19}, among others, are known. Furthermore, many of the geometric properties of these limits are also known \cite{LeGall2,LeGall3,GN,CC19,LGR1,LGR2}. Many of these important results heavily relied on Markovian decompositions of these maps, commonly referred to in the community as their peeling procedure \cite{Cu}. In this work, we focus on Boltzmann maps, as they exhibit the strongest version of the Markov property.

\subsection{Results}
A planar map is a rooted graph embedded on the sphere. In this work, and just for the sake of simplicity, we only work on planar quadrangulations with a boundary, that is, planar maps in which all faces are surrounded by 4 semi-edges \footnote{An edge is composed of two semi-edges, one per side, and a face can see both semi-edges of an edge.}, except possibly the root face. We say that $\q$ is a quadrangulation (with a boundary and) with holes if it is a planar map with boundary and marked faces, called holes, which need not be squares, and all non-marked interior faces are surrounded by $4$ semi-edges.

We can construct a quadrangulation $\BM$ with boundary from a quadrangulation with holes $\q$ as follows. For each hole $h$ of $\q$ we take a quadrangulation with boundary $\BM^\q_h$ with boundary size equal to the length of the hole. Then, we glue to each hole its respective quadrangulation to obtain $\BM$ (for a precise definition, see Section \ref{sec:3} and Figure \ref{fig:gluing}). We call this operation $\q\glue (\BM^\q_h)_h$. When there exist maps $(\BM^\q_h)_h$ such that $\BM= \q\glue (\BM^\q_h)_h$ we say that $\q$ is contained in $\BM$, or $\q\Lalg \BM$.

We study a collection of measures $(\P^\ell)_{\ell \in \N^*}$ such that $\P^{\ell}$ is supported on quadrangulations with boundary of half-perimeter $\ell$. The Markov property is defined as follows.

\begin{defn}[Markov Property]\label{d.MP_Undecorated}
   The collection of probability measures $(\P^\ell)_{\ell\in \N^*}$ satisfies the Markov property if for any deterministic quadrangulation with holes $\q$, $\BM\sim \P^{\ell}$ conditioned on the event $\{\q\Lalg \BM\}$ decomposes as follows
\begin{align*}
\BM= \q\glue (\BM^\q_h)_h,
\end{align*}
where the collection $\BM^\q_h \sim \P^{\text{Per}(h)/2}$ is independent and independent of $\q$, and the operator $\glue$ denotes gluing\footnotemark[\value{footnote}] $\BM^\q_h$ inside each hole $h$, for all holes.  See Figure \ref{fig:holes}.\footnote{See Section \ref{models} for the precise definitions and \Cref{fig:gluing} for an intuitive scheme.}
\end{defn}
\begin{figure}[h!]
    \centering
    \includegraphics[scale=0.8]{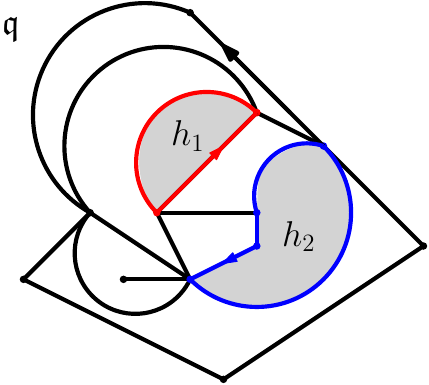}
    \caption{A quadrangulation with a boundary (defined by the root edge in black) and holes $h_1$ and $h_2$ (in grey) with special directed edges.}
    \label{fig:holes}
\end{figure}

The first result of this paper completely characterises the space of rerooting invariant (see \ref{Invariance-root} in Section \ref{models} for a precise definition) Markovian measures on planar quadrangulations with boundary. A more precise version of the following theorem appears in Theorem \ref{t.main_undecorated}.

\begin{thm}\label{t.main_characterisation_no_decoration}
    Let $(\P^\ell)_{\ell\in \N^*}$ be a sequence of probability measures on quadrangulations with boundary that satisfy the Markov property and are invariant under rerooting ($\P^{\ell}$ is invariant after moving the root one unit to the left). Then there exists $0\leq q\leq q_c=1/12$ such that $\P^\ell$ is the Boltzmann probability measure with weight $q$. More precisely, for any quadrangulation with boundary $\q$ of half-perimeter $\ell$,
    \begin{align*}
    \P^\ell(\q) \propto q^{|F(\q)|},
    \end{align*}
    where $|F(\q)|$ is the number of internal faces of $\q$.
\end{thm}
This result may appear similar to Theorem 2 of  \cite{budzinski2021local}. However, in \cite{budzinski2021local}, the authors introduce a property called \textit{weakly Markovian} for infinite triangulations of the plane. Specifically, this property states the following: for an infinite triangulation $T$ and a fixed triangulation $t$ with a hole, the probability that $T$ arises from gluing $t$ with another infinite triangulation $T'$ depends only on: the number of internal vertices of $t$, and the perimeter of the hole in $t$. Notably, this condition is weaker than our Markov property, as it does not specify the distribution of the infinite triangulation $T'$. In our context, this is more closely related to what we denote in Section \ref{sec:3} uniformly distributed on $\Qm^{\ell,f}$, the quadrangulations with half-perimeter $\ell$ and $f$ faces. More formally, the probability of obtaining $T$ as a gluing result from $t_1$ or $t_2$ is the same if $t_1$ and $t_2$ have the same number of internal vertices and the same hole perimeter.
We use this property to prove Theorem \ref{t.main_characterisation_no_decoration}, however we always work with maps that have the Markov property, so our proof techniques are different from \cite{budzinski2021local}.

As mentioned, the Markov property of planar maps has traditionally been presented via the so-called peeling procedure. In this paper, we use an alternative approach based on what we call stopping maps. To define them, we introduce filtrations indexed by quadrangulations with holes $\q$, i.e., a collection of $\sigma$-algebras $(\F_\q)_{\q}$ such that $\q_1\subset \q_2$ implies $\F_{\q_1} \subset \F_{\q_2}$. 

We say that a Boltzmann map $\BM$ is an $\F$-Boltzmann map with parameter $q$ if 
\begin{itemize}
    \item For any quadrangulation with holes $\q$, the event $\{\q\subset \BM\}$ is $\F_{\q}$-measurable;
    \item Conditionally on $\F_\q$ and $\{\q\subset \BM\}$, the collection $(\BM^\q_h)_h$ is independent, and each $\BM^\q_h$ is a Boltzmann map with parameter $q$ and perimeter given by the boundary size of $h$.
\end{itemize}

We say that a random quadrangulation with holes $\SM$ is an $\F$-stopping map for an $\F$-Boltzmann map $\BM$ if almost surely $\SM\subset \BM$ and for any $\q$, the event $\{\SM \subset \q\}$ is $\F_\q$-measurable. We use this setting to characterise random maps with holes that induce a Markovian decomposition.

\begin{thm}\label{t.main_strong_Markov_no_decoration}
    Let $(\F_\q)_{\q}$ be a filtration, $\BM$ be an $\F$-Boltzmann map with parameter $q$, and $\SM$ an $\F$-stopping map. Then there exists a collection of maps $(\BM^\SM_h)_{h}$ indexed by the holes of $\SM$ which satisfies:
    \begin{enumerate}
        \item $\BM = \SM \glue (\BM^\SM_h)_h$;
        \item Conditionally on $\SM$, the maps $(\BM^\SM_h)_{h}$ are independent and distributed as Boltzmann maps with parameter $q$ and perimeter equal to the boundary size of $h$.
    \end{enumerate}
    Furthermore, assume we have a coupling $(\BM, \SM)$ that satisfies (1) and (2). Then there exists a filtration $(\F_\q)_\q$ such that $\BM$ is an $\F$-Boltzmann map and $\SM$ is an $\F$-stopping map.
\end{thm}
This theorem is proved in two parts. The first one is Theorem \ref{t.strong_Markov}, where we show that $\F$-stopping maps satisfy (1) and (2). The second one is Theorem \ref{t.local_implies_stopping}, where we take a coupling $(\BM,\SM)$ and construct a filtration that satisfies the properties.

Theorem \ref{t.main_strong_Markov_no_decoration} provides a characterisation of random maps that induce a Markovian decomposition. One might expect this characterisation is also equivalent to the sets obtained through a (randomised) peeling procedure (see e.g., \cite{Cu}), which discovers the set $\SM$ one face at a time. However, we show in Section \ref{s.not_algorithmic} that this is not the case. There exists a stopping set $\SM$ such that there is no sequence $(\SM_n)$ of maps obtained via peeling steps such that $\SM = \bigcup_n \SM_n$ and  $\SM_n \setminus \SM_{n-1}$ contains at most one face.

\subsection{Ideas of proof}
As we have discussed, there are two main types of theorems in this work. The first one characterises laws that satisfy the Markov property and the second one characterises the (random) submaps that satisfy the Markov property. Remarkably, this paper is completely self-contained, except for the fact that we use that certain sums over quadrangulations with boundaries are finite.

\subsubsection*{Characterisation of laws} Let us first describe the proof strategy in this case. We start with an intermediate step. We first restrict our attention to a subclass of probability measures that satisfy the Markov property, and whose law is prescribed on maps with a fixed number of faces. We require that the law conditioned on this set is the uniform one.

To show that the intermediate step implies being a Boltzmann map, we locally modify the map by adding a face. We show that, for any quadrangulation, the ratio between the probability of the original map and that of its modified version only depends on the length of the outer face. By writing this ratio in two different ways, we deduce that it must be equal to a global parameter~$q$, which corresponds to the Boltzmann weight. Then, we need to show that the Markov property and the rerooting invariance imply the intermediate step. To do this, we use different types of peeling steps to obtain the same map. In this case, we use induction to reduce the number of faces and conclude.

\subsubsection*{Characterisation of submaps} We again split the proof in two parts. First, we define stopping maps and show that they satisfy the Markov property. To do this, we are inspired by both the classical theory of Markov processes and the theory of local sets of the Gaussian free field from the point of view of its filtrations (See Chapter 1 of \cite{Aru}). In our case, filtrations are indexed by possible submaps rather than time or closed sets. Once the right definitions are set in place, many of the proofs closely mirror classical arguments. Thus, the main difficulty is properly choosing the right indexation for the filtrations.

Second, we show that every random submap that induces a Markovian decomposition must in fact be a stopping map. This is done by starting with a Markovian coupling (or ``local map coupling,'' in analogy with the GFF; see \cite{SS, Aru}) and explicitly constructing a filtration $\F$ such that the Boltzmann map is an $\F$-Boltzmann map and the submap is an $\mathcal{F}$-stopping map. The key input here is that the decomposition of a Boltzmann map along a deterministic submap $\q$ can be obtained by first sampling the local part and then decomposing the remainder, using what is left of $\q$.

\subsubsection*{A stopping map that cannot be constructed by any peeling procedure.} Finally, let us discuss how to construct a stopping map $\SM$ that is non-algorithmic. For simplicity, we work with a Boltzmann quadrangulation with half-perimeter $1$. We start from the root edge and we successively take the right-most face (this forms a cycle) in which the same edge is never discovered twice by the peeling procedure. This exploration can also be obtained backwards by starting to explore from the other boundary edge and following the left-most face successively. Now, consider the smallest subcycle (with no repeated faces along the right-most path) along this path that contains the root edge (see Figure \ref{fig:Q} for the drawing in the dual). This is in fact a stopping map but it cannot be obtained via a peeling procedure. To prove the latter fact, one assumes that such a peeling exists and proceeds discovering it one edge at a time. A simple case-by-case argument shows that, apart from the two boundary edges, any attempt to discover another face has positive probability of discovering a face that does not belong to $\SM$.
 
 \medskip

The paper is organised as follows. In Section \ref{sec:3}, we characterise maps that satisfy the weak Markov and strong Markov properties using stopping maps. In Section \ref{sec:LocalMaps}, we show that if a submap induces a Markovian decomposition of a Boltzmann map, then it must be a stopping map for a suitable filtration. Finally, in Section \ref{s.not_algorithmic}, we present an example of a stopping map that cannot be obtained via the peeling procedure.

\subsection*{Acknowledgements}
P.A. and A.S. thank the program "Probabilistic methods in quantum field theory" of the Hausdorff Research Institute for Mathematics where the last parts of this paper were finished. P.A. is supported by ANID-Subdireccion de Capital Humano/Doctorado Nacional/2023-21231096. L.F is partially supported by the Project PEPS JCJC 2025–UMR 5251 (IMB), INSMI. The research of A.S.\,is supported by Centro de Modelamiento Matem\'{a}tico Basal Funds FB210005 from ANID-Chile, FONDECYT regular 1240884 and ERC 101043450 Vortex, and was supported by FONDECYT iniciaci\'on de investigaci\'on N$^o$ 11200085.

\subsection*{AI Declaration}
This paper together with \cite{AFS26} were originally only one paper that was uploaded to ArXiv on May 2025, so before the AI craziness of 2026. Thus, this work benefited of AI only at the level of language checking and a referee like revision. More precisely, no mathematical results were obtained through interaction with AI, but we hope clarity was improved with our interaction with it.

\section{Markov property and associated measures}\label{sec:3}

In this section, we give a theoretical framework for the spatial Markov property associated with a general type of maps defined by \cite{Cu}. We also give a characterisation of the maps that have the Markov property under reasonable hypotheses.
\subsection{Model.}\label{models}
A planar map $\m$ is a finite connected graph that is properly embedded in the sphere  with a distinguished oriented edge $e_r$ called its \textit{root} and whose starting vertex is called the \textit{root} vertex. We denote the following sets associated with $\m$,
\begin{itemize}
    \item $V(\m);$ set of vertices of $\m$,
    \item $E(\m);$ set of edges of $\m$,
    \item $\overrightarrow{E}(\m);$ set of oriented edges (or half-edges) of $\m$,
    \item $F(\m);$ set of faces of $\m$.
\end{itemize}

Furthermore, we call \textit{root face} the face to the right of $e_r$. We define the \textit{perimeter} of $\m$, denoted by $Per(\m)$, as the degree of the root face. 

An important type of map is the map with holes. We say that a map has \textit{holes} if the map is provided with a sequence of distinguished faces $h_1,\dots,h_n$ each one with a marked oriented edge $e_1,\dots,e_n$. We denote $H(\m)$ as the set of all holes of $\m$. Also, we denote $\PM_H$ as the space of rooted planar maps with holes. Additionally, we define the \textit{active boundary} of $\m$, denoted by $Active(\m)$, as the set of edges adjacent to the holes of $\m$. 

From here onward, and for the sake of simplification\footnote{All the results of this paper should extend easily in the context of more general maps.}, we work only with maps that are quadrangulations with a boundary, that is to say all faces are squares except possibly the exterior one. Additionally, we restrict to quadrangulations $\q$ such that the dual of $\q$ forms a connected graph and every vertex of $\q$ is incident to a face. This condition is important because it implies the property of being rigid (see Lemma 4.8 in \cite{angel2003uniform}) for quadrangulations, which is key in order to have a Markov property. In this context and for $\ell, f\in \N$, we define $\Qm^{\ell,f}$ as the set of  quadrangulations with boundary of half-perimeter $\ell$ and $f$ interior faces, $\Qm^{\ell}= \bigcup_{f}\Qm^{\ell,f}$ and  $\Qm= \bigcup_{\ell} \Qm^{\ell}$. Furthermore, the set of quadrangulations with holes is denoted by $\Qm_H$.

Now, we define a way to glue maps. Let $\q_1$ be a rooted planar map with only one hole $h$ and $\q_2$ be a rooted planar map with holes having $Per(\q_2) = deg(h)$. We identify the associated distinguished edge of the hole with the root edge of $\q_2$. Then, following the orientation of this oriented marked edge, we identify the edges (and consequently the vertices) of the boundary of $\q_2$ with the edges of the hole $h$. The resulting map of \textit{gluing $\q_2$ into $h$}, denoted by $\q_1\glue\q_2$, is the map with vertex $V(\q_1)\cup V(\q_2)$ and edges $E(\q_1)\cup E(\q_2)$ counting only once the identified vertices and identified edges, with holes $H(\q_2)$, and with the same boundary as $\q_1$ maintaining the root. In the case that $\q_1$ has more than one hole, we can glue this map with a collection of maps $(\q_h)_{h\in H(\q_1)}$ such that $Per(\q_h) = deg(h)$ and $\q_h$ is glued with the hole $h$. In this case, we denote this new map as $\q_1 \glue (\q_h)_{h\in H(\q_1)}$.
\color{blue}
\begin{figure}[h!]
    \centering
    \includegraphics[scale = 1]{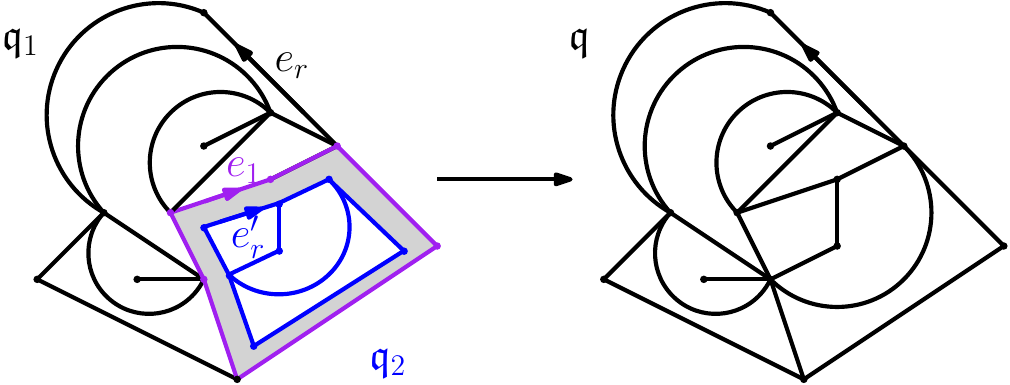}
    \caption{An example of a map $\q$ obtained from the gluing between $\q_1$ and $\q_2$, where  to the left we see the map $\q_1$ with one hole (coloured grey and surrounded by purple) and $\q_2$ the map in blue. The transformation identifies (glues) $e_1$ the marked edge of $\q_1$ and $e_r$ the root edge of $\q_2$ in blue and all the edges following their orientation in the hole of $\q_1$ and the edges of the external face of $\q_2$.}
    \label{fig:gluing}
\end{figure}
\color{black}

Thanks to this operation, we can define an order relation between maps. We say that a map with holes $\q_1$ is a \textit{submap} of a map with holes $\q_2$ with the same root of $\q_1$, denoted by $\q_1\Lalg \q_2$, if there exists a collection of maps with holes $(u_h)$, indexed by the holes of $\q_1$, such that
\begin{itemize}
    \item $Per(u_h) = deg(h)$, for every $h\in H(\q)$,
    \item and, $\q_2 = \q_1\glue (u_h)_{h\in H(\q_1)}$.
\end{itemize}

\begin{rem}\label{rem_bef_pee}
    The collection of maps $(u_h)_{h\in H(\q_1)}$ that needs to be glued to a map $\q_1$ to obtain another map $\q_2$ is uniquely defined\footnote{The property of being rigid assures the uniqueness of the collection $(u_h)_{h\in H(\q_1)}$.} for any map $\q_1 \Lalg \q_2$.
\end{rem}
This relation has a minimal element called the \textit{cemetery}, denoted by $\dagger$, which is the empty map. This map satisfies that $\dagger\Lalg\q$, for any $\q\in\Qm_H$.

The previous order relation allows a way of exploring rooted planar maps in an algorithmic manner\footnote{Recall that we are working with finite planar maps.}, called the \textit{peeling exploration}. Starting from the face next to the root $\mathfrak{e}_0$, the algorithm discovers a face of the map $\q$ in each iteration, giving as a result, a finite collection of maps $(\mathfrak{e}_i)_{i}$, such that
\begin{align}
    \mathfrak{e}_0\Lalg\mathfrak{e}_1\Lalg \dots \Lalg\mathfrak{e}_n = \q.
\end{align}
This exploration depends on a function $\mathcal{A}$, called the \textit{peeling algorithm}, which takes a map with holes $\mathfrak{e}$ and gives an edge $e$ from its active boundary. For a map $\q$ such that $\mathfrak{e}\subseteq \q$ we denote $F_e$ the face in $\q$ that is adjacent to $\q$ that was not adjacent to $e$ in $\mathfrak e$. Then $F_e$ can be of two types.
\begin{itemize}
    \item\textbf{Peeling step of type 1:} The face $F_e$ is not a face of $\mathfrak{e}_i$. Then, $\mathfrak{e}_{i+1}$ is obtained by gluing $F_e$ with $\mathfrak{e}_i$ on $e$. The new active boundary set is obtained from the previous one by suppressing the selected edge $e$ and adding the newly discovered edges in $F_e$. The new active boundary is the result of erasing the selected edge $e$ and adding the rest of the edges of the discovered face $F_e$. In what follows, we denote by \type{$1$}  the peeling step of type 1.
    \item\textbf{Peeling step of type 2:} The face $F_e$ is actually a face of $\mathfrak{e}_i$. In this case, the edge $e$ is identified in $\mathfrak{m}$ with another edge $\overline{e}$ on the boundary of the same hole. The map $\mathfrak{e}_{i+1}$ is obtained from $\mathfrak{e}_i$ by identifying the half-edge $e$ with the half-edge $\overline{e}$. The holes that do not have $ e$ as a half-edge remain unchanged, while the hole that has the edge $e$ is now divided into two holes\footnote{Note that for this to work, we assign the label $1$ to the hole to the left of the half edge.}, $h_1$ and $h_2$, with perimeter $Per(h_1) + Per(h_2) + 2 = Per(\mathfrak{e}_i)$, joined by this edge $e$. Any of those holes that did not have a marked half-edge is given a new marked half-edge starting from $e$ and going in the right-hand direction. we denote by \type{$2,\ell_1,\ell_2$}  the peeling step of type 2 with $Per(h_1) = 2\ell_1$ and $Per(h_2) = 2\ell_2$. Whenever the perimeter of one of the holes is 0, it is seen as hole that has disappeared. We use this notation to divide the perimeter of the original hole into two parts.
\end{itemize}
 we denote by $Peel(\mathfrak{e_i},e,\q)$ to the resulting map of doing a peeling step on $e\in Active(\q) $ from $\mathfrak{e_i}$  (See Figure \ref{fig:peel}).  
 
 We now state some important properties that may be satisfied by the law of a random planar quadrangulation.

\begin{figure}[h!]
    \centering
    \includegraphics[width=5in]{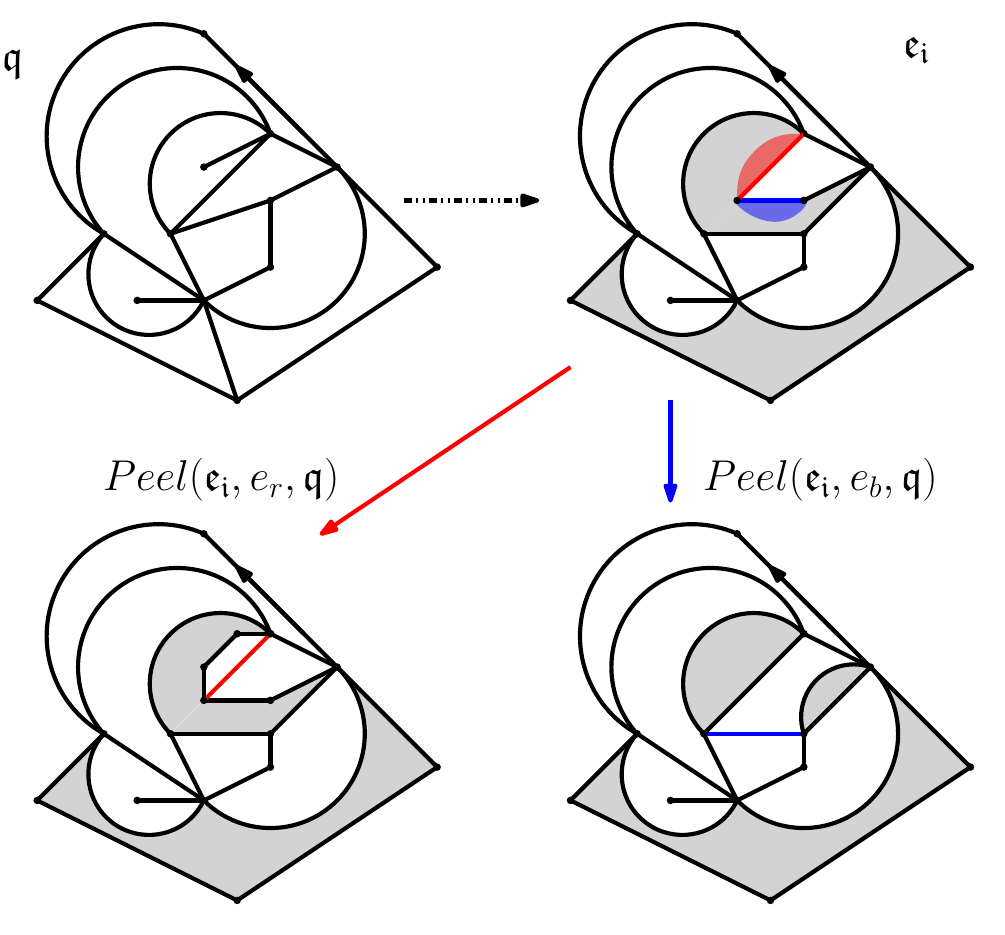}
    \caption{At the top left, the map $\q$ and immediately to its right the map with holes obtained after some peeling iterations. We colored white the discovered map by the peeling steps so far and here the boundary of the grey regions represents the active boundary. At the bottom left the result of a peeling iteration type 1 (\textit{Type $1$}) on $\q$ from $\mathfrak{e_i}$ when doing a peeling step on the red edge $e_r$. And to its right the result of a peeling iteration type 2 (\textit{Type $2$,2,2}) on $\q$ from $\mathfrak{e_i}$ when doing a peeling step on the blue edge $e_b$.}
    \label{fig:peel}
\end{figure}

\begin{enumerate}
    \item\label{Invariance-root} \textbf{(Invariance under rerooting)}
    $\BM$ is \textit{invariant under rerooting} if for any deterministic rooted quadrangulation $\q_1$ and any copy $\q_2$ of $\q_1$ differing only by the position of the root edge, $\BM$ satisfies
    \begin{align*}
            \P(\BM=\q_1) = \P(\BM=\q_2).
    \end{align*}
    \item\label{uniform-fix-faces} \textbf{(Uniformly distributed on $\Qm^{\ell,f}$)} $\BM$ is \textit{uniformly distributed in $\Qm^{\ell,f}$} if for any $\q_1,\q_2 \in \Qm^{\ell,f}$ 
    \begin{align*}
        \P\left(\BM= \q_1\right) = \P\left(\BM= \q_2\right).
    \end{align*}
\end{enumerate}
\begin{rem}\label{rem:unif-re}
    If $\BM$ is uniformly distributed on $\QQ^{\ell,f}$, then it immediately satisfies the property of invariance under rerooting, since the rerooting procedure preserves the number of faces and the half-perimeter.
\end{rem}

Now we present a well known law on random quadrangulations that is key for the results of this paper: the Boltzmann map. For $q> 0$ and $\ell \in \N$, we say that $\BM$ is a \textit{$q$-Boltzmann map} with half-perimeter $\ell$ if the law of $\BM$ is supported on $\QQ^{\ell}$ and is given by 
\begin{align*}
    \P^{\ell}(\BM = \q) =\frac{1}{W_q^{\ell}}q^{|F(\q)|},\qquad \forall \q\in \QQ^{\ell}. 
\end{align*}
Here $W_q^{\ell}$ is the normalising constant. Note that this probability measure only makes sense as long as $q$ is a positive constant such that $W_q^{\ell} < \infty$, which is true for $q\leq \frac{1}{12}$ (see for example \cite{Cu} ). We will only mention the parameter $q$ when necessary, but in general we omit it.

An important result about the Boltzmann maps is the following.
\begin{prop}\label{prop:inv}
    Boltzmann maps are invariant under rerooting (\ref{Invariance-root}) and are uniformly distributed on $\QQ^{\ell,f}$ (\ref{uniform-fix-faces}).
\end{prop}
\begin{proof}
  From Remark \ref{rem:unif-re} it suffices to prove the uniform distribution on $\QQ^{\ell,f}$. Let $\q_1, \q_2\in \QQ^{\ell,f}$, then
    \begin{align*}
        \P^{\ell}(\BM = \q_1) = \frac{1}{W_q^{\ell}}q^{|F(\q_1)|} =  \frac{1}{W_q^{\ell}}q^{|F(\q_2)|} = \P^{\ell}(\BM = \q_2).
    \end{align*}
\end{proof}

\subsection{Markov property: stopping maps.}
In this section, we state the Markov property for Boltzmann maps. We begin by stating the weak Markov property associated with conditioning on an event depending on a deterministic submap, and then, we describe a type of random submaps that also satisfies the Markov property that we call \textit{stopping maps}, as it is folklore we call this version the strong Markov property. In Section \ref{sec:LocalMaps}, we will also show that any submap that satisfies the Markov property is also a stopping map.
\subsubsection{Weak Markov property.}
In this section, we show that Boltzmann maps satisfy the weak Markov Property. This is a well-known result (see for example \cite{angel2015classification,Cu}), even for infinite planar maps; but both for the sake of completeness and because later proofs are going to be based on it we present the proof here.
\begin{thm}\label{t.weakMarkov}
    Let $\BM$ be a Boltzmann map and $\q\in\Qm_{H}$. Then, on the event $\{\q \Lalg \BM\}$, $\BM$ can be decomposed as follows
    \begin{align*}
        \BM = \q \glue (\BM^\q_h)_{h\in H(\q)},
    \end{align*}
    where $(\BM^\q_h)_{h\in H(\q)}$ is a collection of independent $q$ Boltzmann maps with boundary $h$ for every $h\in H(\q)$.
\end{thm}
\begin{proof} 
Denote by $\Qm_H(\q)$ the following set
\begin{align*}
    \Qm_H(\q) &= \left\{(\q_h)_{h\in H(\q)}\subseteq\Qm_H: |\partial \q_h| = Per(h)\text{ for every }h\in H(\q)\right\}.
\end{align*}
Then, for $\p\in\Qm_H$ such that $\q\Lalg\p$, we have the following
\begin{align}\label{Markov_proof}
    \P^{\ell}(\BM = \p|\q\Lalg\BM) &= \frac{\P^{\ell}(\BM = \p,\q\Lalg\BM)}{\P^{\ell}(\q\Lalg\BM)} = \frac{\P^{\ell}\left(\BM = \q\glue(\p_h)_{h\in H(\q)} \right)}{\P^{\ell}(\q\Lalg\BM)}.
\end{align}
Using that the set of possible values of $\BM$ is countable, we can rewrite the denominator as a sum
\begin{align*}
    \P^{\ell}(\q\Lalg\BM) &= \sum_{(\q_{h})_{h\in H(\q)}\in\Qm_H(\q)}\P^{\ell}\left(\BM = \q\glue (\q_h)_{h\in H(\q)}\right)\\
    &= \frac{q^{|F(\q)|}}{W_q^{\ell}}\sum_{(\q_{h})_{h\in H(\q)}\in\Qm_H(\q)}q^{\sum_{h\in H(\q)}|F(\q_h)|}\\
    &=  \frac{q^{|F(\q)|}}{W_q^{\ell}}\prod_{h\in H(\q)}W_q^{\ell(h)},
\end{align*}
where $\ell(h) = Per(h)/2$. Then, replacing in \eqref{Markov_proof}
\begin{align*}
    \P^{\ell}(\BM = \p\mid\q\Lalg\BM) &= \frac{\frac{1}{W_q^{\ell}}q^{|F(\q)| + \sum_{h\in H(\q)}|F(\p_h)|}}{\frac{q^{|F(\q)|}}{W_q^{\ell}}\prod_{h\in H(\q)}W_q^{\ell(h)}} = \prod_{h\in H(\q)}\left(\frac{q^{|F(\p_h)|}}{W_q^{\ell(h)}}\right).
\end{align*}
where $(\p_h)_{h\in H(\q)}$ satisfies that $\p = \q\glue(\p_h)_{h\in H(\q)}$. With this equality we obtain that for every $h\in H(\q)$
\begin{align*}
    \P^{\ell(h)}(\BM_h^{\q} = \q_h\mid\q\Lalg\BM) = \frac{q^{|F(\q_h)|}}{W_q^{\ell(h)}},
\end{align*}
which is the distribution of a $q$-Boltzmann map with boundary $h$. Finally, we write the previous equality as follows
\begin{align*}
    \P^{\ell}(\BM = \q\glue(\q_h)_{h\in H(\q)} \mid\q\Lalg\BM) &=  \prod_{h\in H(\q)}\P^{\ell(h)}(\BM_h^{\q} = \q_h),
\end{align*}
which implies the independence of the collection $(\BM^\q_h)_{h\in H(\q)}$.

\end{proof}

\subsubsection{Filtrations and stopping maps.}
In order to describe the strong Markov property, we define a filtration indexed by planar maps. 
\begin{defn}
    We say that a collection of $\sigma-$algebras $\FF = (\F_{\q})_{\q\in\Qm_H}$ is a \textit{filtration} if it satisfies the property of \textit{Monotonicity},
\begin{align}
    \tag{Monotonicity} \mbox{For every }\q_1,\q_2\in\Qm_H\mbox{ such that }\q_1\Lalg \q_2\mbox{ one has } \F_{\q_1}\subseteq\mathcal{F}_{\q_2}.\label{filtration-monotonicity}
\end{align}
When working with a probability measure $\P$, the filtrations that we will use will also be complete for it, i.e
\begin{align}
    \tag{Completeness} \mbox{ }\F_{\q}\mbox{ is complete with respect to }\P,\mbox{ for every }\q\in\Qm_H.\label{filtration-completesness}
\end{align}
\end{defn}

For the rest of this section, $\BM$ is going to be a Boltzmann map. 
\begin{defn}\label{def_Boltzmann}
    We say that a Boltzmann map $\BM$ is an $\FF$-\textit{Boltzmann map} if it satisfies the following properties
\begin{itemize}[label={--}]
    \item\label{Adaptability}\textit{Adaptability:} for any $\q$ in $\Qm_H$, the event $\{\q\Lalg \BM\}$ is $\F_{\q}-$measurable.
    \item\label{Indep-increments}\textit{Independent increments:} conditionally on $\F_{\q}$ and on the event $\{\q\Lalg \BM\}$, the law of $(\BM^\q_h)_{h\in\mathcal{H}(\q)}$ is that of a collection of independent $q$ Boltzmann maps with boundary $h$ for every $h\in\mathcal{H}(\q)$.
\end{itemize}
\end{defn}
\begin{rem}
   To give an intuition about filtrations in this context, let us remark that the natural filtration of a Boltzmann map $\BM$ is defined as
   \begin{align*}
    \F_{\q} := \overline{\bigvee_{\p\Lalg \q} \sigma\left( \{\p \Lalg \BM\} \right)}^{\P}.
\end{align*}
We use the join operator to denote the $\sigma$-algebra generated by the union, this is important in order to have \textbf{monotonicity} of the filtration. The over-line represents the completion of the $\sigma$-algebra with respect to $\P$, i.e. it gives the $\sigma$-algebra generated when adding all subsets of negligible sets for $\P$. It obviously gives the \textbf{completeness} with respect to $\P$ of the filtration. Firstly, the filtration is built so that the \textbf{adaptability} of $\BM$ is direct, and secondly, the property of \textbf{independent increments} of $\BM$ is a direct consequence of the weak Markov property (Theorem \ref{t.weakMarkov}).
\end{rem}

To obtain a strong Markov property, we need to define a specific class of random submaps. 
\begin{defn}\label{def_stop}
We say that a random planar map $\SM$ is an $\FF$-\textit{stopping map} for an $\FF$-Boltzmann map $\BM$ if
\begin{enumerate}
    \item $\P$-almost surely $\SM\Lalg \BM$,
    \item and, for any $\q\in\Qm_H$ we have that  $\left\{\SM\Lalg \q\right\}\in\F_{\q}$.
\end{enumerate}
\end{defn}

\begin{rem}\label{r.subpeeling}
	An important example of a stopping map is the following. Take a peeling $\mathcal A$ and a stopping time $\tau$ of it. Then, the submap explored up to time $\tau$ is a stopping map. In Section \ref{s.not_algorithmic}, we will see that not all stopping maps can be built this way.
\end{rem}

Let us note that for a random map with holes $\SM$, the events $\Qm_{\subseteq \q}$ characterise the $\sigma$-algebra generated by $\SM$.
\begin{prop}\label{prop:sigma_algebra}
     Let $\SM$ be a random map and for $\q\in \Qm_H$ consider the events $\Qm_{ \Lalg \q} := \{\SM\subseteq \q\}$. Then,
     \begin{align}\label{event_eq}
         \sigma\left(\SM\right) = \sigma\left(\Qm_{\Lalg \q}:\q\in\Qm_H\right).
     \end{align}
     Recall that as $\SM$ takes only countably many values, its $\sigma$-algebra is generated by the events $\{\SM=\p\}$ for $\p\in\Qm_H$.
\end{prop}
\begin{proof}
    For $\supseteq$: Notice that for every $\q\in\Qm_H$ we have that 
    \begin{align*}
    \Qm_{\Lalg \q } = \bigcup_{\substack{\p \in \Qm:\\ \p\Lalg \q}} \{\SM=\p\}.
    \end{align*}Then $\sigma\left( \Qm_{\Lalg \q}:\q\in\Qm_H\right)\subseteq \sigma\left(\SM\right).$
    
    For $\subseteq$:  For $\p\in \Qm_{H}$, we can write
    \begin{align*}
    \{\SM=\p\} = \Qm_{\Lalg \p} \cap \left( \bigcup_{\substack{\q\in \Qm_H:\\\q\Lalg \p, \q\neq \p}} \Qm_{\Lalg \q} \right)^c.
    \end{align*}Then, $\{\SM=\p\}\in \sigma\left( \Qm_{\Lalg \q}:\q\in\Qm_H\right)$, which gives the result.
\end{proof}

Now, we state some important properties associated with the stopping maps.
\begin{prop}\label{prop.stopping}
    Let $\BM$ be an $\FF-$Boltzmann map and $\SM$ an $\FF-$stopping map for $\BM$. Then, for $\q\in\QQ_H$, the following maps are $\F_{\q}-$measurable,
    \begin{enumerate}
        \item $\SM\mathbbm{1}_{\SM\Lalg \q}:= \begin{cases}
            \SM, &\mbox{ if }\SM\Lalg \q,\\
            \dagger, &\mbox{ otherwise,}
        \end{cases}$
        \item $\SM^{\wedge \q}:= \begin{cases}
            \SM, &\mbox{ if }\SM\Lalg \q,\\
            \q, &\mbox{ if }\SM\not\Lalg\q \mbox{ and } \q\Lalg \BM,\\
            \BM, &\mbox{ if }\SM\not\Lalg\q \mbox{ and }  \q\not\Lalg \BM.
        \end{cases}$
    \end{enumerate}
    Furthermore, $\SM^{\wedge \q}$ is also an $\FF$-stopping map for $\BM$.
\end{prop}
\begin{proof}
Let $\q\in\QQ_{H}$, we prove each result individually.
\begin{enumerate}
    \item Thanks to Proposition \ref{prop:sigma_algebra}, it suffices to show that $\{\SM\mathbbm{1}_{\SM\Lalg \q}\Lalg \q_1\}\in\F_{\q}$ for any $\q_1\in\QQ_{H}$. There are two possibilities: \\
    $\bullet$ The map $\SM$ is not included in $\q$ : 
    \begin{align*}
\{\SM\mathbbm{1}_{\SM\Lalg \q}\Lalg \q_1\}\cap \{\SM \nsubseteq \q\} = \{\dagger\Lalg \q_1\}\cap \{\SM \nsubseteq \q\}= \{\SM \nsubseteq \q\} \in \F_{\q}.
    \end{align*}
    $\bullet$ The map $\SM$ is included in $\q$: 
    \begin{align*}
        \{\SM\mathbbm{1}_{\SM\Lalg \q}\Lalg \q_1\} \cap \{\SM\Lalg\q\} &= \{\SM\Lalg \q_1\}\cap \{\SM\Lalg\q\} = \{\SM\Lalg \q_1\wedge \q\}\in \F_{\q_1\wedge\q} \subseteq \F_{\q},
\end{align*}

where $\q_1\wedge \q$ denotes the biggest map with holes that is contained\footnote{Note that this map is unique as if you take two of them you can join them so that they are still contained in both $\q$ and $\q_1$.} in both $\q$ and $\q_1$.
    \item Similarly as before, we have three cases:\\
    $\bullet$ The map $\SM$ is included in $\q$
    \begin{align}\label{e.Mm_contenido}
    \{\SM^{\wedge \q}\Lalg \q_1\}\cap \{\SM\Lalg\q\}  = \{\SM\Lalg \q_1\wedge\q\}\in\F_{\q_1\wedge\q}\subseteq\F_{\q}.
    \end{align} 
    $\bullet$ The map $\SM$ is not included in $\q$ and $\q$ is included in $\BM$: 
    \begin{align}\label{e.Mm_nocontenido_1}
\{\SM^{\wedge \q}\Lalg \q_1\}\cap \{\SM \nsubseteq \q\} \cap\{\q\Lalg\BM\} = \{\q\Lalg \q_1\}\cap \{\SM\not\Lalg \q\}\cap\{\q\Lalg\BM\} \in \F_{\q}.
    \end{align}
    $\bullet$ The map $\SM$ is not included in $\q$ and $\q$ is not included in $\BM$:
    \begin{align}\label{e.Mm_nocontenido_2}
\{\SM^{\wedge \q}\Lalg \q_1\}\cap \{\SM \not\Lalg\q\}\cap\{\q\not\Lalg\BM\}  = \{\q\Lalg \q_1\}\cap \{\SM \nsubseteq \q\}\cap\{\q\not\Lalg\BM\} \in \F_{\q}.
    \end{align}
In order to prove that $\SM^{\wedge\q}$ is an $\FF$-stopping map, we need to prove that $\P$-almost surely $\SM^{\wedge \q}\Lalg\BM$ and $\{\SM^{\wedge \q}\Lalg\q_1\}\in\F_{\q_1}$ for any $\q_1\in\Qm_H$. Notice that the first condition comes directly by the definition of $\SM^{\wedge \q}$. For the second condition, notice that this is clear when $\SM$ is contained in $\q$ from the right-hand side of \eqref{e.Mm_contenido}, since $\F_{\q_1\wedge\q}\subseteq\F_{\q_1}$. For the case when $\SM$ is not contained in $\q$, i.e. for the event in \eqref{e.Mm_nocontenido_1} and \eqref{e.Mm_nocontenido_2}, we have that if $\q\Lalg \q_1$ this implies $\F_\q\subset \F_{\q_1}$ so we have the result; and if not, then the event $\{\q \Lalg \q_1\}$ is empty and thus it trivially belongs to $\F_{\q_1}$.
\end{enumerate}
\end{proof}
\begin{defn}\label{stopping-sigma}
We define the $\sigma$-algebra associated with an $\FF$-stopping map $\SM$ as follows
\begin{align*}
    \F_{\SM} := \left\{\left.\Theta\in\bigvee_{\q\in\Qm_H}\F_{\q}\right| \Theta\cap \{\SM\Lalg\q\}\in\F_{\q},\mbox{ for any }\q\in\Qm_H\right\}.
\end{align*}
\end{defn}

We have the following properties for this $\sigma$-algebra.
\begin{prop}\label{p.Basic_properties_F_M}
Let $\BM$ be an $\F$-Boltzmann map and $\SM, \SM_1, \SM_2$ be $\FF$-stopping maps for $\BM$. Then,
    \begin{enumerate}
        \item  $\SM$ is $\F_{\SM}-$measurable. 
        \item If a.s. $\SM_1\Lalg \SM_2$, we have that $\F_{\SM_1}\subseteq\F_{\SM_2}$.
    \end{enumerate}
\end{prop}
\begin{proof}$\ $
\begin{enumerate}
    \item It suffices to prove that the events of the form $\Theta = \{\SM\Lalg \q\}$ belong to $\F_{\SM}$, for all $\q\in\Qm_H$.\\
    Let $\q_1\in\Qm_H$ be arbitrary. It follows that
    \begin{align*}
        \Theta \cap \{\SM\Lalg \q_1\} = \{\SM\Lalg \q\wedge \q_1\}\in\F_{\q\wedge\q_1} \subseteq\F_{\q_1},
    \end{align*}
    from which the conclusion follows.
    \item Let $\Theta\in\F_{\SM_1}$ and $\q\in \Qm_H$. We will show that $\Theta\cap\{\SM_2\Lalg \q\} \in \F_{\q}$.
    \begin{align*}
        \Theta\cap\{\SM_2\Lalg \q\} &\stackrel{a.s.}{=} (\Theta\cap\{\SM_1\Lalg \q\}) \cap \{\SM_2\Lalg \q\} \cap \{\SM_1\Lalg \SM_2\}.
    \end{align*}
    The first term on the right-hand side belongs to $\F_{\q}$, from the hypothesis; and the second term also does because $\SM_2$ is an $\FF$-stopping map. Finally the last term belongs to $\F_\q$, since it is the complement of a negligible event and the $\sigma$-algebra $\F_{\q}$ is complete with respect to $\P$ (the probability associated with Boltzmann maps).
\end{enumerate}
\end{proof}
Now we have all the elements we need to state the spatial Markov property by means of (random) stopping maps, this is what we call \textbf{Strong Markov property}.

Here we show that Boltzmann maps satisfy the strong Markov property.
\begin{thm} \label{t.strong_Markov}
    Let $\BM$ be an $\FF$-Boltzmann map and $\SM$ an $\FF$-stopping map for $\BM$. Then, $\BM$ can be decomposed as follows
    \begin{align*}
        \BM = \SM \glue (\BM^{\SM}_h)_{h\in H(\SM)},
    \end{align*}
    where, conditional on $\F_{\SM}$, $(\BM^{\SM}_h)_{h\in H(\SM)}$ is a collection of independent $q$-Boltzmann maps with boundary $h$ for every $h\in H(\SM)$.
\end{thm}
In order to prove the strong Markov property we need to prove the following lemma.

\begin{lemma}
    For any real integrable random variable $X$, we have that
    \begin{align}\label{e.Conditional_expectation}
        \E(X\mathbbm{1}_{\SM=\q}|\F_{\SM}) = \E(X|\F_{\q})\mathbbm{1}_{\SM=\q}.
    \end{align}
\end{lemma}
\begin{proof}
    We first check that the right-hand side of \eqref{e.Conditional_expectation} is $\F_\SM$-measurable.   To do that, take $\p \in \Qm_H$ and  a Borel set $A\subseteq\R$ that does not contain $0$. It suffices to show that the following event is $\F_\p$-measurable    \begin{align}\label{e.measurability_cond_exp}
        \{\E(X|\F_{\q})\mathbbm{1}_{\SM=\q}\in A\}\cap \{\SM\subseteq \p\} &=\begin{cases}
        \emptyset  & \text{ if } \q\not\Lalg \p\\
        \{\E(X|\F_{\q}) \in A\} \cap \{\SM= \q\} & \text{ if } \q \Lalg \p,
        \end{cases}
    \end{align}
The first case follows trivially since $\emptyset \in \F_\p$. When $\q\subset \p$, note that $\{\E(X|\F_{\q}) \in A\}\in \F_\q$ and, thanks to equation \eqref{event_eq}, we have that the event $\{\SM=\q\}$  belongs to $\F_{\q}$. Then $\{\SM=\q\}\in \F_\q$ (from \ref{p.Basic_properties_F_M} (1)), and that $\F_\q\subset\F_\p$.

We are left to show that the left-hand side of \eqref{e.Conditional_expectation} satisfies the integral property of the conditional expectation. To do that take  $\Theta\in\F_{\SM}$ and compute
\begin{align*}
    \E(X\mathbbm{1}_{\SM=\q}\mathbbm{1}_{\Theta}) &= \E\left[ \E(X\mathbbm{1}_{\SM=\q}\mathbbm{1}_{\Theta}|\F_{\q})\right]  = \E\left[\E(X|\F_{\q})\mathbbm{1}_{\SM=\q}\mathbbm{1}_{\Theta}\right],
\end{align*}
which gives the result.
\end{proof}

Now, we are ready to prove the strong Markov property.

\begin{proof}[Proof of Theorem \ref{t.strong_Markov}]
    Since $\SM$ is an $\F$-stopping map, we know that a.s. $\SM\Lalg\BM$. This implies that there exists a collection of maps $(\BM^{\SM}_h)_{h\in\mathcal{H}(\SM)}$, such that
\begin{align*}
    \BM = \SM \glue (\BM^{\SM}_h)_{h\in H(\SM)}.
\end{align*}
Denote by $P$ the set of all the possible values for $\SM$ and notice that this set is countable. Take $h\in H(\SM)$ and  $(f_h)_{h\in H(\SM)}$ a collection of measurable functions taking values in $[0,1]$ measurable with respect to $H(\SM)$. Then,

\begin{align*}
    \E\left[\prod_{h\in H(\SM)}f_h(\BM^{\SM}_h)\Big|\F_{\SM}\right] &= \sum_{\q\in P}\E\left[\prod_{h\in H(\SM)}f_h(\BM^{\q}_h)\mathbbm{1}_{\SM = \q}\Big|\F_{\SM}\right]\\
    &= \sum_{\q\in P}\prod_{h\in H(\q)}\E\left[f_h(\BM^{\q}_h)\Big|\F_{\q}\right]\mathbbm{1}_{\SM = \q}.
\end{align*}

We conclude since $\{\SM=\q\}$ is $\F_\SM$-measurable (from \ref{p.Basic_properties_F_M} (1)) and that conditionally on $\F_\q$ the map $\BM^{\q}_h$ has the law of a Boltzmann map for any $\q\in P$, this is nothing else than the weak spatial Markov property.
\end{proof}

\begin{rem}\label{r.weak->strong} In this section, we have shown that a map satisfying the weak Markov property will also satisfy the strong Markov property. Together with Remark \ref{r.subpeeling}, this implies that for any map satisfying the weak Markov property, the submap explored by a peeling exploration at a stopping time induces a Markovian decomposition of the map. 
		
		Finally, when we say that the distribution of a random quadrangulation satisfies the Markov property we mean that it satisfies the weak Markov property, and thus the strong one.
\end{rem}

\subsection{Characterisation of random quadrangulations satisfying the Markov Property.}
Here, we present the main result of this section that is to characterise the planar maps that satisfy the Markov property. In order to do this, we need to state the Markov property for a collection of measures $(\P^{\ell})_{\ell\in\N^*}$ supported on quadrangulations with half-perimeter $\ell$.

\begin{defn} We say that $(\P^{\ell})_{\ell\in\N^*}$ satisfies the \textit{Markov property}, if, for any deterministic quadrangulation with holes $\q$, we can describe the conditional law  $\P^{\ell}(\cdot \mid \{\q\Lalg \BM\})$ as
    \begin{align*}
        \BM = \q\glue (\BM^{\q}_h)_{h\in H(\q)}.
    \end{align*}
Here, $(\BM^{\q}_h)_{h\in H(\q)}$ is a collection of independent maps with law $\P^{\ell(h)}$ for any $h\in H(\q)$.
\end{defn}

This definition together with Remark \ref{r.weak->strong} immediately implies explicit formulae for the peeling procedure (definition after Remark \ref{rem_bef_pee}).
\begin{rem}
    The Markov property implies the following decomposition. If the peeling step on the root edge of $\q$ is a type 1 peeling, we have that
\begin{align}\label{Decomposition type 1}
    \P^{\ell}(\BM = \q) = \P^{\ell}(\BM = \q\mid\type{1})\P^{\ell}(\type{1})  = \P^{\ell+1}(\BM = \hat{\q})\P^{\ell}(\type{1}),
\end{align}
where $\hat{\q}$ is the map remaining after the peeling of the root of the quadrangulation $\q$. We also have an analogous formula when the peeling step of the root edge is of type 2
\begin{align}\label{Decomposition type 2}
    \P^{\ell}(\BM = \q) = \P^{\ell_1}(\BM = \q_1)\P^{\ell_2}(\BM = \q_2)\P^{\ell}(\type{2,\ell_1,\ell_2}),
\end{align}
where $\q_1$ and $\q_2$ are the maps remaining after the peeling step of the root edge of $\q$. 

These previous equalities imply the following decomposition for the event $\{\BM = \q\}$
\begin{align}\label{Decomposition}
    \P^{\ell}(\BM = \q) &= \prod_{i = 1}^{\sharp \text{ steps}} \P^{\ell_i}(\type{j_i}),
\end{align}
where $\ell_i$ and $j_i$ depend on the structure of $\q$ and on the step $i$ of the peeling procedure.
\end{rem}

Now, we are ready to state the main result of this section, which characterises all the quadrangulations having the Markov property.

\begin{thm}\label{t.main_undecorated}
The following three conditions are equivalent for a sequence of measures $(\P^{\ell})_{\ell\in\N^*}$ supported on quadrangulations with a boundary of size $2\ell$
\begin{enumerate}
    \item There is a $q\geq 0$ such that for all $\ell\in\N^*$, $
    \P^{\ell}$ corresponds to the law of a $q$-Boltzmann map,
    \item $(\P^{\ell})_{\ell\in\N^*}$ has the Markov property and for each $\ell$ the measure $\P^\ell$ is uniformly distributed on $\Qm^{\ell,f}$ for every $f$,
    \item $(\P^{\ell})_{\ell\in\N^*}$ has the Markov property and for each $\ell$ the measure $\P^\ell$ is invariant under rerooting.
\end{enumerate}
\end{thm}

\begin{proof}
From Property \ref{prop:inv} we know that $(1)\Rightarrow (2)\Rightarrow (3)$. We are going to prove $(2)\Rightarrow (1)$ and $(3)\Rightarrow (2)$. Notice that in both implications we can assume $(3)$. This is important because during this proof we are going to use more $(3)$ than $(2)$.
\begin{itemize}
    \item $(2)\Rightarrow (1)$: 
    Let us first assume that for some $\ell_0$ we have that $\P^{\ell_0}(\type{1})=0$, then $\mathbb{P}_{\ell_0}$ is supported on maps without faces. By the Markov property, all the laws $\mathbb{P}^\ell$ obtained by filling holes must have the same support. Consequently, for every $\ell$, $\mathbb{P}^\ell$ is the uniform measure on trees of half-perimeter $\ell$, which coincides with the $0$-Boltzmann measure. 
    
    Now, let us assume that $\P^{\ell}(\type{1})>0$ for every $\ell$.  For $\ell\in\mathbb{Z}_+$ and $\q\in\Qm_H$, with half-perimeter $\ell$, we define $q(\ell,\q)$ as follows\footnote{Note that the quantity $q(\ell,\q)$ is well defined since $\P^{\ell}(\type{1})>0$ plus the fact that $\P^\ell $ is uniformly distributed on $\Qm^{\ell,f}$, implies that for any quadrangulation $\q$ with a boundary of size $\ell$, $\P^{\ell}(\BM=\q)>0$.}
\begin{align*}
    q(\ell,\q) := \frac{\P^{\ell}(\BM = \overline{\q})}{\P^{\ell}(\BM = \q)},
\end{align*}
where $\overline{\q}\in\Qm_H$ is a copy of $\q$ but gluing at the root edge a new face as drawn in Figure \ref{f.tilde m}.

\begin{figure}[ht!]	   	   
			\centering
			\includegraphics[width=5in]{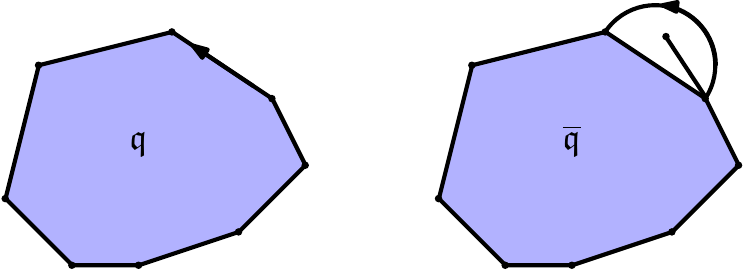}
			\caption {At the root edge we glue a new face and we declare the new external edge as the new root edge, with the right orientation in order to keep the external zone as the new external face. With this procedure $\overline{\q}$ satisfies that $|F(\overline{\q})| = |F(\q)| + 1$ and $|\partial\overline{\q}| = |\partial \q|$.}
			\label{f.tilde m}
\end{figure}

To prove the Theorem we just need to show that $q(\cdot,\cdot)$ is constant. For this, we first prove that for any $\ell\in \N^*$, $q(\ell,\cdot)$ is constant. Note that (2) implies that $q(\ell,\q)$ does not depend on the structure of $\q$, only on the number of its faces. We now show that it also does not depend on the number of faces of $\q$. Note that to obtain $\q$ from $\overline{\q}$ through a peeling exploration, we must first perform a type 1 peeling step, followed by a type $2,0,\ell$ peeling step. Let $\widehat{\q}$ the resulting map of the first peeling step of $\overline{\q}$. Then

\begin{align*}
     q(\ell,\q) &= \frac{\P^{\ell}(\BM = \overline{\q})}{\P^{\ell}(\BM = \q)}\\
     &= \frac{\P^{\ell+1}(\BM = \widehat{\q})\P^{\ell}(\type{1})}{\P^{\ell}(\BM = \q)}\\
     &= \P^{\ell}(\type{1})\P^{\ell+1}(\type{2,0,\ell}).
\end{align*}
From this we deduce that $q(\ell,\cdot)$ is constant, consequently we now note $q(\ell) := q(\ell,\q)$ where $\q\in\Qm_H$ is any quadrangulation with boundary $\ell$.

We claim that the definition of $q(\ell,\q)$ and property (2) imply
\begin{align}\label{casi_boltz}
    \P^{\ell}(\BM = \q) = \frac{q(\ell)^{|F(\q)|}}{W_{q(\ell)}^{\ell}},
\end{align}
where $W_{q(\ell)}^{\ell}<\infty$ is the normalisation constant. To prove the claim, since $\P^\ell$ is uniformly distributed over the maps with the same number of faces, we compare $\q$ with the map $\q^*$ which is the result of iterating the procedure in Figure \ref{f.tilde m} a number $F(\q)$ of times in a map which is equal to a path of length $\ell$ (rooted in any edge). From this we deduce the claim. 

Returning to the proof, we use \eqref{casi_boltz} to show that $q$ does not depend on $\ell$.
\begin{align*}
    q(\ell) &= \frac{\P^{\ell}(\BM = \overline{\q})}{\P^{\ell}(\BM = \q)}\\
    &= \frac{\P^{\ell+1}(\BM = \widehat{\q}) \P^{\ell}(\type{1})}{\P^{\ell}(\BM = \q)}\\
    &= \left(\frac{q(\ell+1)}{q(\ell)}\right)^{|F(\q)|+1}\frac{W_{q(\ell)}^{\ell}}{W_{q(\ell+1)}^{\ell+1}}\P^{\ell}(\type{1}).
\end{align*}
Notice that the right side depends on the number of faces of $\q$ but the left side does not. Doing the same for two different maps with number of faces equal to $1$ and $2$ we obtain the following equality 
\begin{align*}
    \frac{q(\ell+1)}{q(\ell)} = \left(\frac{q(\ell+1)}{q(\ell)}\right)^{2} \implies q(\ell)= q(\ell +1),
\end{align*}
so $q(\cdot)$ does not depend on the half-perimeter $\ell$ and therefore we conclude that
\begin{align*}
    \P^{\ell}(\BM = \q) = \frac{q(1)^{|F(\q)|}}{W_{q(1)}^{\ell}},
\end{align*}
which is the distribution of a $q(1)$-Boltzmann map.

    \item $(3)\Rightarrow (2)$:  This implication follows from \eqref{Decomposition}. The fact that $\P^\ell$ is invariant under re-rooting implies that \eqref{Decomposition} is true for any possible deterministic exploration of $\q$.

    We now show that $\P^\ell$ puts the same mass on quadrangulations with the same number of faces. We do this by induction on the number of faces. The base case is when the map has $0$ internal faces, meaning that it is a tree. In that case there is a way of exploring it just by peeling at each stage the first leaf encountered when following the root edge. This implies that if $\mathbf t$ is a tree with $\ell$ edges 
    \begin{align*}
    \P^{\ell}(\BM= \mathbf t) = \prod_{i=1}^{\ell} \P^{\ell-i+1}(\type{2, 0,\ell-i}).
    \end{align*}
    Thus, all trees have the same probability.
    
    Now, for the inductive step, we assume that for any $\ell\in \N$ all quadrangulations with $n$ faces and half-perimeter $\ell$ have the same probability for $\P^\ell$. Take a quadrangulation $\q$ with $n+1$ faces and half-perimeter of size $\ell$, then there exists an internal face that shares an edge with the external face. Doing a peeling step on this edge and using \eqref{Decomposition type 1} together with invariance under rerooting implies that
    \begin{align*}
    \P^{\ell}(\BM=\q)= \P^{\ell+1}(\BM=\hat \q) \P^\ell(\type{1}),
    \end{align*}
    where $\hat \q$ is the peeling step of type 1 of $\q$. By induction hypothesis, $\P^{\ell+1}(\BM=\hat \q)$ does not depend on $\hat \q$ so we conclude.
\end{itemize}
\end{proof}

\section{Local maps are stopping maps.}\label{sec:LocalMaps}
In this section, we characterise all random submaps of a Boltzmann map that satisfy the Markov property. Informally, we show that the only ones are stopping maps. To do this, we first need to define the concept of a local map, namely, the maps that induce a Markovian decomposition.
\begin{defn}
We say that a pair $\left(\BM,\SM\right)$ is a \textit{local map} if,
\begin{itemize}
    \item $\BM$ is a Boltzmann map;
    \item $\P$-almost surely $\SM\Lalg\BM$;
    \item and, conditionally on $\SM$, $(\BM_h^{\SM})_{h\in H(\SM)}$ is a collection of independent Boltzmann maps with perimeter $|h|$ for every $h\in H(\SM)$. 
\end{itemize}
\end{defn}
This definition is inspired by an analogous one for the Gaussian free field originally introduced by Schramm and Sheffield in Section 3.2 of \cite{SchS}.

The main result of this section is that local sets are stopping sets. To our knowledge, even in the context of the Gaussian free field, this result has not been established. However, in that setting, all the necessary tools were already available; see, for example, Lemma~3.3 of \cite{SchS} or Chapter~1 of \cite{Aru}.

\begin{thm}\label{t.local_implies_stopping}
    Let $\left(\BM,\SM\right)$ be a local map. Then, there exists a filtration $\FF = (\FF_{\q})_{\q\in\Qm_H}$ such that $\BM$ is an $\FF-$Boltzmann map and $\SM$ is an $\FF-$stopping map for $\BM$.
\end{thm}

Take $\left(\BM,\SM\right)$ a local map. We now note that conditionally on $\SM$, $(\BM_{\tilde{h}}^{\SM})_{\tilde{h}\in H(\SM)}$ is a collection of independent Boltzmann maps with perimeter $|\tilde{h}|$ for every $\tilde{h}\in H(\SM)$. Now, we use the weak Markov property inside each hole $\tilde{h}$, to deduce that on the event $\{\SM\Lalg\q\Lalg \BM\}$ we can decompose $\BM^\SM_{\tilde{h}}$ as follows
    \begin{align}\label{e.weak_Markov_SM}
        \BM_{\tilde{h}}^{\SM} = \BM_{\tilde{h}}^{\SM}|_{\q} \glue \left(\left(\BM_{\tilde{h}}^{\SM}\right)_h^{\q}\right)_{h\in H\left(\BM_{\tilde{h}}^{\SM}|_{\q}\right)}.
    \end{align}
Here  conditionally on $\SM$, $\left(\BM_{\tilde{h}}^{\SM}\right)_h^{\q}$ is a collection of independent Boltzmann maps indexed by the holes $h$ of $\q$. Furthermore, the conditional law $\left(\BM_{\tilde{h}}^{\SM}\right)_h^{\q}$ is that of a Boltzmann map with perimeter $|h|$.

To prove Theorem \ref{t.local_implies_stopping}, we explicitly define the filtration
\begin{align}\label{e.def_filt}
        \F_{\q} = \overline{\bigvee_{\p\Lalg\q} \sigma\left(\{\p\Lalg\BM\},\{\SM\Lalg \p\}, \SM^{\wedge \p}\1_{\p\Lalg\BM} \right)}^{\P}.
\end{align}
We then can reduce to using the monotone class theorem to compute specific expected values. To make the argument easier to follow and to separate the conceptual ideas from technical computations, we first present two lemmas that carry the ideas of the proof. The first one concerns the computation of a certain conditional law.
\begin{lemma}\label{lem.local_stopping_1}
Let $\left(\BM,\SM\right)$ be a local map. Then for $\q\in\Qm_H$ and a collection $(f_h)_{h\in H(\q)}$ of real, measurable, bounded functions, we have
\begin{align*}
    &\E\left[\left.\prod_{h\in H(\q)}f_h(\BM_h^{\q})\right|\SM,\{\SM\Lalg\q\Lalg\BM\}\right]\1_{\SM\Lalg\q\Lalg\BM} =\prod_{h\in H(\q)}\E^{\frac{|h| }{2}}\left[f_h(\BM)\right ] \1_{\SM\Lalg\q\Lalg\BM}.
\end{align*}
Here, $\BM_h^{\q}$ is the one that appears from the weak Markov property of $\BM$.
\end{lemma}

\begin{proof}
   The first observation is that if $\tilde h$ is a hole of $\SM$ that contains a hole $h$ of $\q$, then $\BM_h^{\q}= (\BM^{\SM}_{\tilde h})^{\q}_h$. Using this together with \eqref{e.weak_Markov_SM} and the comment that follows it, we have that
\begin{align*}
    &\E\left[\left.\prod_{h\in H(\q)}f_h(\BM_h^{\q})\right|\SM,\{\SM\Lalg\q\Lalg\BM\}\right]\1_{\SM\Lalg\q\Lalg\BM}\\
    = &\prod_{h\in H(\q)}\E\left[\left. f_h((\BM_{\tilde h}^{\SM})_h^{\q})\right|\SM,\{\SM\Lalg\q\Lalg\BM\}\right]\1_{\SM\Lalg\q\Lalg\BM}.
\end{align*}
Finally, observe that the conditional law of $(\BM_{\tilde h}^{\SM})_h^{\q}$ does not depend on the geometric structure of $\SM$, and is in fact that of a Boltzmann map inside $h$ with perimeter $|h|$. In consequence
\begin{align*}
    \E\left[\left. f_h\left ((\BM_{\tilde h}^{\SM})_h^{\q}\right )\right|\SM,\{\SM\Lalg\q\Lalg\BM\}\right]\1_{\SM\Lalg\q\Lalg\BM}   = \E^{\frac{|h|}{ 2}}\left[\left. f_h(\BM_h^{\q})\right|\q\Lalg\BM\right]\1_{\SM\Lalg\q},
\end{align*}
which gives the result.
\end{proof}

The second lemma presents the basic case of the computation needed in the proof of Theorem \ref{t.local_implies_stopping}.
\begin{lemma}\label{lem.local_stopping_2}
Let $\left(\BM,\SM\right)$ be a local map and take $\q\in\Qm_H$. Additionally, fix $f_{\q}, (f_h)_{h\in H(\q)}$ and $(h_{p})_{\p\Lalg\q}$ a collection of real, measurable and bounded functions and $\Lambda(\q)$ a collection of submaps of $\q$. Then, we have the following
\begin{align}\label{e.stopping_2}
    &\E\left[f_{\q}(\q)\1_{\SM\Lalg\q\Lalg\BM}\left (\prod_{\p\in \Lambda(\q)}\1_{\SM\Lalg\p\Lalg\BM}h_{\p}(\SM^{\wedge\p})\right ) \left( \prod_{h\in H(\q)}f_h\left(\BM_h^{\q}\right )\right)\right] \\
   \nonumber &= \E\left[f_{\q}(\q) \1_{\SM\Lalg\q\Lalg\BM}\left (\prod_{\p\in \Lambda(\q)}\1_{\SM\Lalg\p\Lalg\BM}h_{\p}(\SM^{\wedge\p})\right ) \left( \prod_{h\in H(\q)}\E^{\frac{ |h| }{2 }}\left[f_h(\BM_h^{\q})|\q\Lalg\BM\right]\right )\right].
\end{align}
\end{lemma}
\begin{proof}
    We apply the tower property to  \eqref{e.stopping_2} by conditioning on  $(\SM,\{\SM\Lalg\q\Lalg\BM\})$. All the terms inside the integral are measurable, except for $\prod_{h\in H(\q)}f_h\left(\BM_h^{\q}\right )$. We conclude from Lemma \ref{lem.local_stopping_1}. 
\end{proof}
Finally, we have all the tools needed to prove Theorem \ref{t.local_implies_stopping}.

\begin{proof}[Proof of Theorem \ref{t.local_implies_stopping}]
Let us verify that the filtration defined in \eqref{e.def_filt} satisfies the conclusions of the theorem.
\begin{itemize}
    \item \underline{$\SM$ is an $\FF$-stopping map.} This follows directly from the definition of the filtration $\FF$.
    \item \underline{The event $\{\q\Lalg\BM\}$ is $\F_{\q}$-measurable.} This follows directly from the definition of the filtration $\FF$.
    \item \uline{Conditionally on $\F_{\q}$ and the event $\q\subseteq \BM$, $(\BM_h^{\q})_{h\in H(\q)}$ is a collection of independent Boltzmann maps with boundary $h$ for every $h\in H(\q)$.} This is the heart of the proof. Thanks to the monotone functional class theorem, it is enough to prove that for any collection of real bounded measurable functions $f_{\q},(f_h)_{h\in H(\q)}, (h_{\p})_{\p\Lalg\q}$ and any collection $\Lambda(\q)$ of submaps of $\q$
\begin{align}\label{e.main_equality_local_to_stopping}
    &\E\left[f_{\q}(\q)\1_{\q\Lalg\BM}\1_{\{\SM\Lalg\q\}^{c_\q}}\left (\prod_{\p\in \Lambda(\q)}\1_{\{\p\Lalg\BM\}^{\tilde c_\p}}\1_{\{\SM\Lalg \p\}^{c_\p}}h_{\p}(\SM^{\wedge\p})\right ) \left( \prod_{h\in H(\q)}f_h\left(\BM_h^{\q}\right )\right)\right] = \\
    \nonumber& \E\left[f_{\q}(\q)\1_{\q\Lalg\BM}\1_{\{\SM\Lalg\q\}^{c_\q}}\left (\prod_{\p\in \Lambda(\q)}\1_{\{\p\Lalg\BM\}^{\tilde c_\p}}\1_{\{\SM\Lalg \p\}^{c_\p}}h_{\p}(\SM^{\wedge\p})\right )  \left( \prod_{h\in H(\q)}\E^{\frac{|h| }{2 }}\left[f_h(\BM_h^{\q})|\q\Lalg\BM\right]
    \right )\right].
\end{align}

Here $c_\q$, $c_\p$ and $\tilde c_\p$ $\in \{1,c\}$ are binary variables saying whether we should or should not take the complement operation. Note that we have already proven the equality in Lemma \ref{lem.local_stopping_2} when all of this binary variables take value 1 (i.e., no complements are taken). We have to do it now when all of this may take the complement. First note that if at least one $\tilde c_p$ takes value $c$ the term inside the expected value is always $0$, as $\p\subseteq \q \subseteq \BM$, thus we can assume they always take value $1$. Now, note that
\begin{align*}
\prod_{\substack{\p\in \Lambda(\q)\\ c_p=c}}\1_{\{\SM\Lalg \p\}^{c}}h_{\p}(\SM^{\wedge\p}) = \prod_{\substack{\p\in \Lambda(\q)\\ c_p=c}}(1-\1_{\{\SM\Lalg \p\}})h_{\p}(\p).
\end{align*}
This reduces \eqref{e.main_equality_local_to_stopping} to a sum of terms of the form of \eqref{e.stopping_2} by properly redefining $\Lambda(\q)$ and
\begin{align*}
    \tilde f_\q(\q):= f_\q(\q)\prod_{\substack{\p\in \Lambda(\q)\\ c_p=c}}h_{\p}(\p).
\end{align*}
\end{itemize}
\end{proof}

\section{A stopping map that is not algorithmic.}\label{s.not_algorithmic}

In this section, we work in the context of Boltzmann quadrangulations as in the beginning of Section \ref{sec:3}. In Remark \ref{r.subpeeling}, we check that every algorithmic exploration can be seen as a stopping map. Now we prove that the converse is false: not all stopping maps can be obtained by algorithmic explorations as defined in \cite{Cu}. The example we present here is inspired by Lemma 7.7 of \cite{MS1}. Let us first describe the dual of a quadrangulation since it is easier to present this counterexample in the dual instead of the primal quadrangulation.

We denote by $\q^{\dagger}$ the dual of the map $\q$, which corresponds to the map where the vertices represent the faces of $\q$, and the edges are dual to those in $\q$ (see Figure \ref{fig:dual}). Notice that if the map is rooted, so is the dual in the associated dual oriented edge. Additionally, if the map has holes, the dual will contain special vertices corresponding to these holes, where for each marked edge, there will be a corresponding dual not oriented marked edge. 

\begin{figure}[H]
    \centering
    \includegraphics[width=0.3\textwidth]{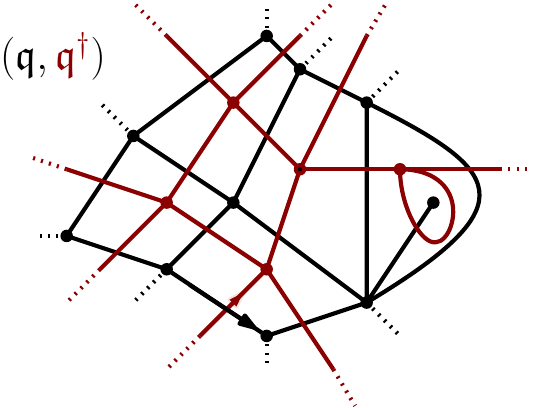}
    \caption{An image of the map $\q$ and $\q^{\dagger}$. Here $\q$ is locally represented in black and $\q^\dagger$ in dark red.}
    \label{fig:dual}
\end{figure}

Let us briefly describe the peeling exploration from the point of view of the dual maps. In each step, instead of  discovering a new face we discover a new vertex  together with the knowledge of how many dual edges are connected to it (see Figure \ref{fig:dual_peel}).

\begin{figure}[h!]
    \centering
    \includegraphics[scale=0.6]{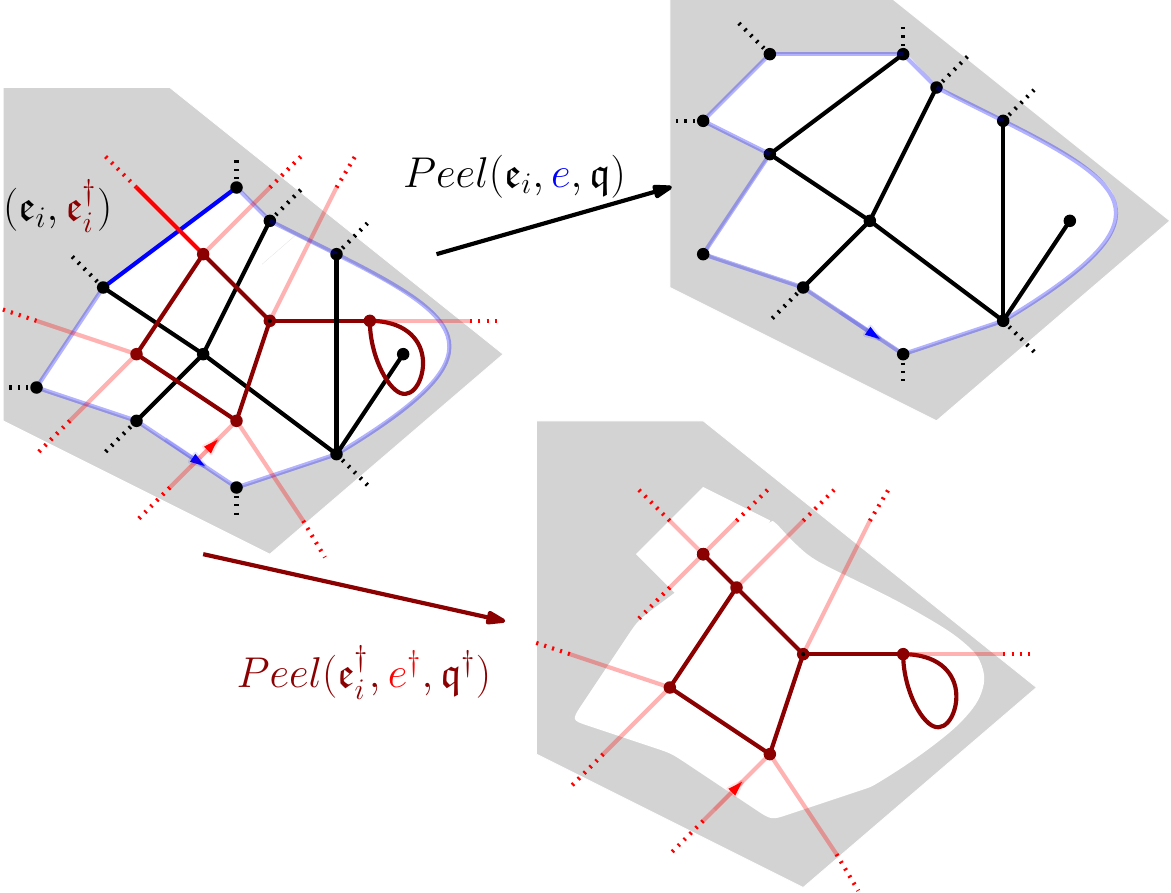}
    \caption{We present to the left a stage of a peeling $\mathfrak{e}_i$ described in section \Cref{sec:3}, where the white area is the one discovered by the peeling so far, seen from the map $\q$ and from the dual map $\q^\dagger$, which reads $\mathfrak{e}_i^\dagger$. We present in blue the active edges and in red the dual associated active edges. To the right we apply one peeling step, when choosing the solid blue active edge and correspondingly the solid red edge. When a face $f$ is discovered in the primal map $\q$, the dual peeling step discovers one vertex $f^{\dagger}$ and the half-edges adjacent to it, which are the dual edges associated with the boundary edges of the face $f$. }
    \label{fig:dual_peel}
\end{figure}

In this section and unless said otherwise, we work on the dual of a quadrangulation with boundary (or a quadrangulation with holes) but we skip the notation $\q^{\dagger}$ and just use $\q$ to not overload the notation.

To start we define on the dual of a random quadrangulation $\QQ$ with boundary two (i.e. under the measure $\P^{1}$) what we call the right process as follows: starting from the root vertex, we peel the root edge, to then explore the map by taking the rightmost dual edge that has not been visited up to the first (and only) return to the root vertex.
We label the dual vertices visited by the right process in chronological order, we identify the point that is the first self-intersection of the right process and we suppress the submap visited between the first and last visit of this point. We define as $\SM$ the map obtained. 

Formally, define $x_0$ as the root vertex and let $x_i$ be the dual vertices visited when considering the rightmost dual edge up to $\tau_E^R$ the first time this process returns the root vertex. We define 
\begin{align*}
	\tau_I^R &= \inf\{i\in \cro{1,\dots,\tau_E^R} : x_i = x_j \text{ for some }j\neq i \in \cro{1,\dots,\tau_E^R}\}\\
	\tau^R &= \sup\{ i\in \N : x_i = x_{\tau_I^R}\}.
\end{align*}

Finally, the map $\SM$ is the map discovered when following the sequence of edges in the cycle $x\cro{0,\tau_I^R}\cup x\cro{\tau^R,\tau^R_E}$. This is analogous to running the right process up to $\tau_I^R$ and then running the left process (the same as the right process but with the leftmost edge) up to hitting $x_{\tau_I^R}$ for the first time. For an example of the stopping map $\SM$ see Figure \ref{fig:Q}

\begin{figure}[h!]
    \centering
    \includegraphics[scale= 0.6]{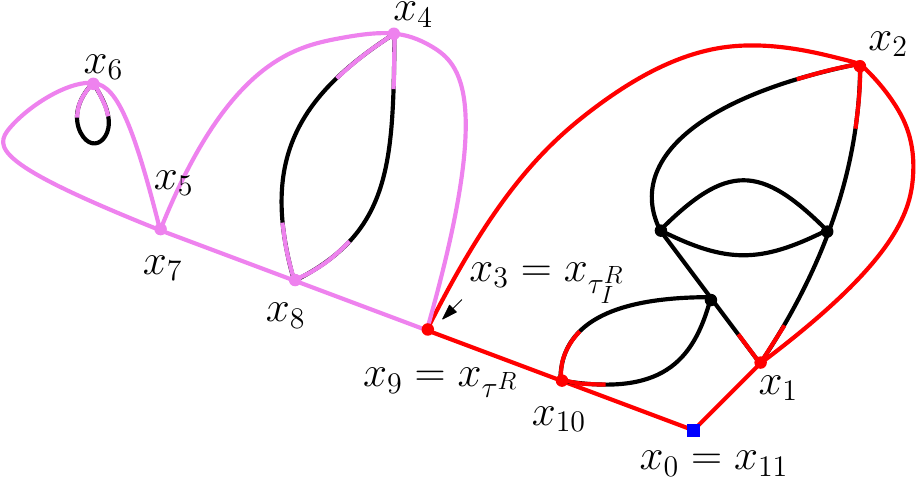}
    \caption{We present the dual of a quadrangulation with exterior face of degree 2, whose dual vertex is represented by the blue square. We follow the right process by labelling the vertices that are visited by it in chronological order and we express in red and fuchsia the map discovered by this process. The first self intersecting point of the right process in this example is $x_3$, meaning that $\tau_I^R=3$ and the last time it is visited is $\tau^R = 9$. The stopping map $\SM$ is given by the red sub-structure as it is given by the exploration of $x\cro{0,\tau_I^R}\cup x\cro{\tau^R,\tau^R_E}$.}
    \label{fig:Q}
\end{figure}

Notice that the complete rightmost exploration is algorithmic, so stopping it and resuming it for the last steps is what will make it a stopping map and not an algorithmic one.
\begin{prop}
    Let $\BM$ be an $\FF$-Boltzmann map for $\ell=1$. The random map $\SM$ is an $\FF$-stopping map for $\FF$ the natural filtration of $\BM$, thus it satisfies the Markov property.
\end{prop}

\begin{proof}
    To prove that it is an $\FF$-stopping map two conditions have to be checked : 
    \begin{enumerate}
    \item $\P$-almost surely $\SM \Lalg \BM$ : this is clear from the definition of $\SM$.
    \item For any $\q\in\Qm_H$ we have that  $\left\{\SM\Lalg \q\right\}\in\F_{\q}$ :  this follows since
    \[
        \left\{\SM\Lalg \q\right\} = \underbrace{\bigcup_{\p\Lalg \q} \underbrace{\{\SM\Lalg \p\}\cap \{ \p \Lalg \BM \}}_{\in \F_\p\subset\F_\q}}_{\in \F_\q},
    \]
    where the event $\{\SM\Lalg\p\}$ for a given $\p\in \QQ_H$ is deterministic and trivial on the event $\{\p \Lalg \BM\}$: run the right process and left process on $\p$ from its root and stop them at the first hole they hit. If the right and left process intersect, then $\SM\Lalg \p$ if not  $\SM \not \Lalg \p$.
\end{enumerate}
\end{proof}

\begin{prop}
     Let $\BM$ be an $\FF$-Boltzmann map with $\ell=1$. Then, the map $\SM$ cannot be obtained by an algorithmic peeling.
\end{prop}

\begin{proof}
    In order to prove this assertion it will be enough to show that with positive probability an algorithmic peeling fails to discover $\SM$ for the possible peeling steps. Se prove this using a branching diagram where we choose to present algorithmic peeling stages that happen with positive probability and such that the leaves (final stages) fail to discover $\SM$. This can be thought as an adversarial argument, where we present from each stage a possible path to fail with positive probability.
    We choose to work conditioned on the event $E$ that $\SM$ has more than three edges.
    
    Without loss of generality we start by doing a peeling step on the edge to the right of the root vertex (root face in the primal map), since the peeling the edge to the left is symmetric to this case. The following figure will be useful to explain our proof.
    
    \begin{figure}[h!]
        \centering
        \includegraphics[width = 0.7\textwidth]{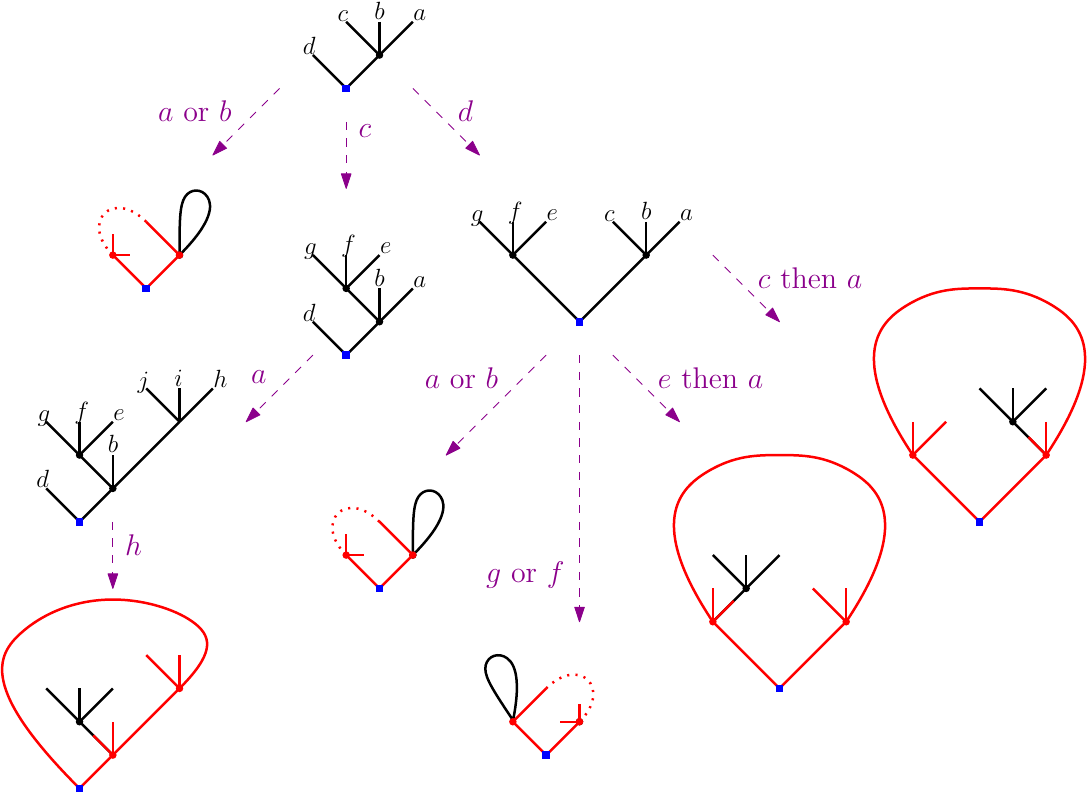}
        \caption{Possible peeling steps with positive probability where the peeling does not properly discover $\SM$. Vertices are represented by dots with exception of the root which is presented with a blue square. We name the edges available to be peeled and in purple the peeling step we apply. Finally, in the tree leaves, we present the map discovered at the end of each sequence of steps where the red submap corresponds to $\SM$. We see at the end of each exploration that they fail to discover $\SM$, since the final structures contain part or the whole $\SM$ with some extra edges and/or vertices.}
        \label{fig:arb}
    \end{figure}
    
    We explain this diagram in what follows and we abbreviate as P.P. the expression ``positive probability''.
    \begin{enumerate}
        \item \textbf{Peeling step on the edge $a$ or $b$:} There is P.P. to identify $a$ with $b$.  In such a case we have peeled the edges in the cycle $x\cro{\tau_I^R,\tau^R}$ which are the ones that are not peeled to obtain $\SM$.
        \item \textbf{Peeling step on the edge $c$:} There is P.P. to discover a new vertex with three new edges to be discovered $e,f,g$. And then by \textbf{peeling step on $a$} there is P.P. to discover a new vertex with three edges $h, i ,j$ and after by \textbf{peeling step on $h$} there is P.P. to identify $h$ with $d$. Here we have peeled the edges in the cycle $\SM$ together with some others (See figure \ref{fig:arb}). 
        \item \textbf{Peeling step on the edge $d$:} We have to discover a new vertex $C$ since we conditioned on $E$, with three edges to be discovered $e,f,g$. If we decide to peel $g$, $f$, $a$ or $b$, the identifications $g$ with $f$ and $a$ with $b$ may happen with positive probability and the conclusion is similar to $(1)$; and if we decide to peel $e$ or $c$ we can discover a new vertex $D$ and copy the strategy in $(2)$.
    \end{enumerate}
\end{proof}

\begin{rem}
In the preceding counterexample we could have worked with arbitrary $\ell$. This is a consequence of the following: suppose you start with $\ell\neq 1$, and apply the peeling procedure until you reach a hole with perimeter 2; at this point we can apply the same arguments as before.
\end{rem}

\bibliographystyle{alpha} 
\bibliography{biblio}

@Preamble{ " \newcommand{\noop}[1]{} " }

@article{BGR19,
	author = {Erich Baur and Gr{\'e}gory Miermont and Gourab Ray},
	title = {{Classification of scaling limits of uniform quadrangulations with a boundary}},
	volume = {47},
	journal = {The Annals of Probability},
	number = {6},
	publisher = {Institute of Mathematical Statistics},
	pages = {3397 -- 3477},
	year = {2019},
	doi = {10.1214/18-AOP1316},
	URL = {https://doi.org/10.1214/18-AOP1316}
}

@inproceedings{CC19,
	title={Self-avoiding walks on the UIPQ},
	author={Caraceni, Alessandra and Curien, Nicolas},
	booktitle={Sojourns in Probability Theory and Statistical Physics-III: Interacting Particle Systems and Random Walks, A Festschrift for Charles M. Newman},
	pages={138--165},
	year={2019},
	organization={Springer}
}

@article{BetM,
	title={Compact {Brownian} surfaces {I}: {Brownian} disks},
	author={Bettinelli, J{\'e}r{\'e}mie and Miermont, Gr{\'e}gory},
	journal={Probability Theory and Related Fields},
	volume={167},
	number={3-4},
	pages={555--614},
	year={2017},
	publisher={Springer}
}

@article{BG,
  author  = {Bouttier, J{\'e}r{\'e}mie and Di Francesco, Philippe and Guitter, Emmanuel},
  title   = {Planar maps as labeled mobiles},
  journal = {Electron. J. Combin.},
  volume  = {11},
  number  = {1},
  pages   = {Paper No. R69},
  year    = {2004}
}

@phdthesis{SCH98,
	title={Conjugaison d'arbres et cartes combinatoires al{\'e}atoires},
	author={Schaeffer, Gilles},
	year={1998},
	school={Universit\'{e} Bordeaux 1}
}

@article{Cu,
  title={Peeling random planar maps},
  author={Curien, Nicolas},
  journal={Saint-Flour lecture notes},
  year={2019}
}

@article{SchS,
  title={A contour line of the continuum {Gaussian} free field},
  author={Schramm, Oded and Sheffield, Scott},
  journal={Probability Theory and Related Fields},
  volume={157},
  number={1},
  pages={47--80},
  year={2013},
  publisher={Springer}
}

@phdthesis{Aru,
  title={The geometry of the Gaussian free field combined with SLE processes and the KPZ relation},
  author={Aru, Juhan},
  year={2015},
  school={Ecole normale sup{\'e}rieure de lyon-ENS LYON}
}

@article{tutte,
  author  = {Tutte, W. T.},
  title   = {A census of planar maps},
  journal = {Canad. J. Math.},
  volume  = {15},
  pages   = {249--271},
  year    = {1963}
}

@article{LeGall,
author = {Jean-Fran{\c{c}}ois Le Gall},
title = {{Uniqueness and universality of the Brownian map}},
volume = {41},
journal = {The Annals of Probability},
number = {4},
publisher = {Institute of Mathematical Statistics},
pages = {2880 -- 2960},
year = {2013},
doi = {10.1214/12-AOP792},
URL = {https://doi.org/10.1214/12-AOP792}
}

@article {Miermont,
    AUTHOR = {Miermont, Gr\'egory},
     TITLE = {The {B}rownian map is the scaling limit of uniform random
              plane quadrangulations},
   JOURNAL = {Acta Math.},
  FJOURNAL = {Acta Mathematica},
    VOLUME = {210},
      YEAR = {2013},
    NUMBER = {2},
     PAGES = {319--401},
      ISSN = {0001-5962,1871-2509},
   MRCLASS = {60D05 (05C80 60F17 60G57)},
  MRNUMBER = {3070569},
MRREVIEWER = {David\ A.\ Croydon},
       DOI = {10.1007/s11511-013-0096-8},
       URL = {https://doi.org/10.1007/s11511-013-0096-8},
}

@article{LeGall2,
  title={The topological structure of scaling limits of large planar maps},
  author={Le Gall, Jean-Fran{\c{c}}ois},
  journal={Inventiones mathematicae},
  volume={169},
  pages={621--670},
  year={2007},
  publisher={Springer}
}

@article {LeGall3,
    AUTHOR = {Le Gall, Jean-Fran{\c{c}}ois},
     TITLE = {Geodesics in large planar maps and in the {B}rownian map},
   JOURNAL = {Acta Math.},
  FJOURNAL = {Acta Mathematica},
    VOLUME = {205},
      YEAR = {2010},
    NUMBER = {2},
     PAGES = {287--360},
      ISSN = {0001-5962,1871-2509},
   MRCLASS = {60J65 (60D05)},
  MRNUMBER = {2746349},
MRREVIEWER = {Xiaowen\ Zhou},
       DOI = {10.1007/s11511-010-0056-5},
       URL = {https://doi.org/10.1007/s11511-010-0056-5},
}

@article{GN,
 author = {Gurel-Gurevich, Ori and Nachmias, Asaf},
 title = {Recurrence of planar graph limits},
 fjournal = {Annals of Mathematics. Second Series},
 journal = {Ann. Math. (2)},
 issn = {0003-486X},
 volume = {177},
 number = {2},
 pages = {761--781},
 year = {2013},
 language = {English},
 doi = {10.4007/annals.2013.177.2.10},
 zbMATH = {6156621},
 Zbl = {1262.05031}
}

@article{CLG,
  title={The {Brownian} plane},
  author={Curien, Nicolas and Le Gall, Jean-Fran{\c{c}}ois},
  journal={Journal of Theoretical Probability},
  volume={27},
  number={4},
  pages={1249--1291},
  year={2014},
  publisher={Springer}
}

@article{CM,
author = {Curien, Nicolas and Miermont, Grégory},
title = {Uniform infinite planar quadrangulations with a boundary},
journal = {Random Structures \& Algorithms},
volume = {47},
number = {1},
pages = {30-58},
doi = {https://doi.org/10.1002/rsa.20531},
url = {https://onlinelibrary.wiley.com/doi/abs/10.1002/rsa.20531},
eprint = {https://onlinelibrary.wiley.com/doi/pdf/10.1002/rsa.20531},
year = {2015}
}

@article{krikun,
  title={Local structure of random quadrangulations},
  author={Krikun, Maxim},
  journal={arXiv preprint math/0512304},
  year={2005}
}

@article{LGR1,
  title={Spatial Markov property in {Brownian} disks},
  author={Gall, Jean-Fran{\c{c}}ois Le and Riera, Armand},
  journal={arXiv preprint arXiv:2302.01138},
  year={2023}
}

@article{LGR2,
  title={Peeling the {Brownian} half-plane},
  author={Le Gall, Jean-Fran{\c{c}}ois and Riera, Armand},
  journal={arXiv preprint arXiv:2404.18489},
  year={2024}
}

@article {SS,
    AUTHOR = {Schramm, Oded and Sheffield, Scott},
     TITLE = {Contour lines of the two-dimensional discrete {G}aussian free
              field},
   JOURNAL = {Acta Math.},
  FJOURNAL = {Acta Mathematica},
    VOLUME = {202},
      YEAR = {2009},
    NUMBER = {1},
     PAGES = {21--137},
      ISSN = {0001-5962,1871-2509},
   MRCLASS = {60J67 (60D05 60G17 60K35)},
  MRNUMBER = {2486487},
MRREVIEWER = {Julien\ Dub\'edat},
       DOI = {10.1007/s11511-009-0034-y},
       URL = {https://doi.org/10.1007/s11511-009-0034-y},
}

@article{budzinski2021local,
  title={Local limits of uniform triangulations in high genus},
  author={Budzinski, Thomas and Louf, Baptiste},
  journal={Inventiones mathematicae},
  volume={223},
  pages={1--47},
  year={2021},
  publisher={Springer}
}

@article{MS1,
  title={Imaginary geometry {I: interacting SLEs}},
  author={Miller, Jason and Sheffield, Scott},
  journal={Probability Theory and Related Fields},
  volume={164},
  pages={553--705},
  year={2016},
  publisher={Springer}
}

@article{angel2015classification,
  title={Classification of half-planar maps},
  author={Angel, Omer and Ray, Gourab},
  year={2015}
}

@article{angel2003uniform,
  title={Uniform infinite planar triangulations},
  author={Angel, Omer and Schramm, Oded},
  journal={Communications in mathematical physics},
  volume={241},
  number={2},
  pages={191--213},
  year={2003},
  publisher={Springer}
}

@misc{AFS26,
  author        = {Araya, Pablo and Fredes, Luis and Sep{\'u}lveda, Avelio},
  title         = {The Markov Property on Decorated Metric Planar Maps},
  howpublished  = {Preprint, arXiv:2611.25081},
  year          = {2026},
  eprint        = {2611.25081},
  archiveprefix = {arXiv}
}
\appendix

\end{document}